\documentclass[12pt]{amsart}
\usepackage{amsbsy,amsmath,amscd,amsfonts,amsgen,amsopn,amssymb,amstext,amsthm,fullpage,setspace,indentfirst}
\usepackage[all,cmtip]{xy}
\newtheorem*{thma}{Theorem A}
\newtheorem*{corb}{Proposition B}
\newtheorem*{thmc}{Theorem C}
\newtheorem*{corc2}{Theorem D} 
\newtheorem*{thmd}{Theorem E}
\newtheorem*{thmf}{Theorem F}
\newtheorem*{thmg}{Theorem G}
\newtheorem*{core1}{Proposition H}
\newtheorem*{core2}{Proposition J}
\newtheorem{thm}{Theorem}[section]
\newtheorem{prop}[thm]{Proposition}
\newtheorem{cor}[thm]{Corollary}
\newtheorem{lem}[thm]{Lemma}
\begin{document}
\title{Residuation of Linear Series\\
and The Effective Cone of $C_{d}$}
\author{Yusuf Mustopa}
\begin{abstract}
We obtain new information about divisors on the $d-$th symmetric power $C_{d}$ of a general curve $C$ of genus $g \geq 4.$  This includes a complete description of the effective cone of $C_{g-1}$ and a partial computation of the volume function on one of its non-nef subcones, as well as new bounds for the effective and movable cones of $C_{d}$ in the range $\frac{g+1}{2} \leq d \leq g-2.$  We also obtain, for each $g \geq 5,$ a divisor on $C_{g-1}$ with non-equidimensional stable base locus.   

For a general hyperelliptic curve $C$ of genus $g,$ we obtain a complete description of the effective cone of $C_{d}$ for $2 \leq d \leq g$ and an integral divisor on $C_{g-1}$ which has non-integral volume whenever $g$ is not a power of 2.
\end{abstract}
\maketitle
\section{Introduction}
Let $C$ be a smooth complex projective algebraic curve, and let $d \geq 1$ be an integer.  The \textit{$d$-th symmetric power} $C_{d}:=C^{d}/{\mathcal{S}_{d}}$ is a smooth $d-$dimensional complex projective variety which is a fine moduli space parametrizing effective divisors of degree $d$ on $C.$  In this paper we study the cone of effective divisors on $C_{d}$ and its refinements.          

There are two natural divisor classes on $C_{d}.$  For each $p \in C$ the image of the embedding $$i_{p}:C_{d-1} \hookrightarrow C_{d}, \hspace{0.1cm} D' \mapsto D'+p$$ is an ample divisor on $C_{d}$ whose numerical class (which is independent of $p$) will be denoted by $x.$  The pullback of a theta-divisor on $\textrm{Pic}^{d}(C)$ via the natural map $$a_{d}:C_{d} \rightarrow \textrm{Pic}^{d}(C), \hspace{0.1cm} D \mapsto \mathcal{O}_{C}(D)$$ is a nef divisor on $C_{d}$ whose numerical class will be denoted by $\theta.$  (When $2 \leq d \leq g,$ we also have that $\theta$ is big.)  The classes $x$ and $\theta$ are linearly independent in the real N\'{e}ron-Severi space $N^{1}_{\mathbb{R}}(C_{d}),$ and when $C$ is general $N^{1}_{\mathbb{R}}(C_{d})$ is generated by $x$ and $\theta.$ 

We prove the following:
\begin{thma}
\label{main1}
Let $C$ be a general nonhyperelliptic curve of genus $g \geq 4.$  For each $d \leq g-1,$ define $$r_{g,d}=1+\displaystyle\frac{g-d}{g^{2}-dg+(d-2)}.$$  
\item[(i)]The class $\theta-r_{g,d}x$ on $C_{d}$ is $\mathbb{Q}-$effective for $d \leq g-1.$
\item[(ii)]The class $\theta-r_{g,g-1}x$ on $C_{g-1}$ spans a boundary ray of the effective cone of $C_{g-1}.$
\end{thma}
This result gives a new bound for the effective cone of $C_{d}$ in the range $\frac{g}{2}+1 \leq d \leq g-1.$  We will say more about the case where $d<\frac{g}{2}+1$ later in this introduction.  

Theorem 3 in \cite{Kou} says that the diagonal class $$\Delta=2(-\theta+(g+d-1)x)$$ (which parametrizes effective divisors of degree $d$ having multiplicity) spans a boundary ray of the effective cone of $C_{d}$ for all $d \geq 2.$  Combining (ii) of Theorem A with that result yields 
\begin{corb}
\label{effcone1}
If $C$ is a general nonhyperelliptic curve of genus $g \geq 4,$ the effective cone of $C_{g-1}$ is spanned by the half-diagonal class $-\theta+(2g-2)x$ and the class $\theta-(1+\frac{1}{2g-3})x.$
\end{corb}

We obtain more refined information in two distinct (but related) directions.  The \textit{stable base locus} $\textbf{B}(D)$ of a $\mathbb{Q}-$Cartier divisor $D$ on a projective variety $X$ is the set-theoretic intersection of the base loci of the linear systems $|mD|$ (where $m$ varies over all positive integers for which $mD$ is Cartier).  The codimension of $\textbf{B}(D)$ is a rough measure of the size of $D;$ for instance, $\textbf{B}(D)=\emptyset$ if $D$ is ample and $\textbf{B}(D)=X$ if $D$ is not pseudoeffective.

The stable base locus of a divisor is not a numerical invariant in general.  However, the so-called \textit{stable} divisors (see Section \ref{baseloc} for the definition) have stable base loci which are numerical invariants, and as such we can speak of stable classes in $N^{1}_{\mathbb{Q}}(X)$.  The \textit{movable cone} of a smooth projective variety $X$ is the closure of the convex cone in $N^{1}_{\mathbb{R}}(X)$ spanned by classes of divisors on $X$ whose stable base locus has no divisorial component.     
 
\begin{thmc}
\label{main2}
Let $C$ be a general nonhyperelliptic curve of genus $g \geq 4,$ and let $d \leq g-1.$ Then the class $\theta-x$ on $C_{d}$ is stable with stable base locus $$C^{1}_{d}=\{D \in C_{d}:\dim{|D|} \geq 1\}.$$  In particular, since $C^{1}_{d}$ is of codimension at least 2, the class $\theta-x$ lies in the interior of the movable cone of $C_{d}.$ 
\end{thmc}

This theorem follows from the more general Theorem \ref{stab_classes}, which when combined with Lemma \ref{nef_small_diag} yields

\begin{corc2}
If $C$ is a general nonhyperelliptic curve of genus $g \geq 5,$ there exists $s_{g} > 1+\frac{1}{g^{2}-g-1}$ such that the the class $\theta-sx$ on $C_{g-1}$ is stable with a non-equidimensional stable base locus whenever $1+\frac{1}{g^{2}-g-1} < s < s_{g}$.  
\end{corc2}   

We now turn to a different way of measuring the size of a divisor.  The \textit{volume} of a $\mathbb{Q}-$Cartier divisor $D$ on an $n-$dimensional projective variety $X$ is $$\textnormal{vol}_{X}(D)=\limsup_{m}\frac{n! \cdot h^{0}(\mathcal{O}_{X}(mD))}{m^{n}}$$ where the limit superior is taken over all positive integers $m$ for which $mD$ is Cartier.  When $D$ is big, this number may be thought of as the "moving self-intersection" of $D$ (see, for instance, Theorem 11.4.11 in \cite{Laz2}).  Note that $\textnormal{vol}_{X}(D) > 0$ precisely when $D$ is big, and that when $D$ is ample, $\textnormal{vol}_{X}(D)=D^{n}$ by Serre vanishing and asymptotic Riemann-Roch.  

It can be shown that $\textnormal{vol}_{X}(D)$ is independent of the numerical class of $D,$ and that the real-valued function $\textnormal{vol}_{X}$ on $N^{1}_{\mathbb{Q}}(X)$ extends to a continuous real-valued function on $N^{1}_{\mathbb{R}}(X).$  (More recently, $\textnormal{vol}_{X}$ has been shown to be $\mathcal{C}^{1}$ in \cite{BFJ}).  
\begin{thmd}
\label{main3}
Let $C$ be a general nonhyperelliptic curve of genus $g \geq 4.$  Then for $t \in [0,1+\frac{1}{g^{2}-g-1}],$ $$\textnormal{vol}_{C_{g-1}}(\theta-tx)=\displaystyle\sum_{k=0}^{g-1}\binom{g-1}{k}\frac{g!}{(k+1)!}t^{k}(1-t)^{g-1-k}.$$
\end{thmd}

We pause to discuss the proofs of the theorems listed thus far.  Recall that for a smooth projective curve $C$ and two positive integers $r$ and $d,$ there exists a fine moduli variety $G^{r}_{d}(C)$ parametrizing linear series (complete or otherwise) of degree $d$ and dimension $r$ on $C.$  In particular, $G^{0}_{d}(C)$ is canonically isomorphic to $C_{d}.$   

When $2 \leq d \leq g-1,$ there is a birational map $$\widetilde{\tau}:G^{g-d-1}_{2g-2-d}(C) \dashrightarrow C_{d}$$ defined by taking a complete linear series $|\mathcal{L}|$ to the unique element of its residual series $|K_{C}\otimes\mathcal{L}^{-1}|.$  This map is the moral heart of the paper and the technical heart of the proofs of Theorems A,C, D, and E.  It follows from Gieseker's Theorem that $\widetilde{\tau}$ is an isomorphism in codimension 1 of smooth projective varieties when $C$ is a general curve of genus $g$, and as a result we can use $\widetilde{\tau}$ to transfer information about divisors from $G^{g-d-1}_{2g-2-d}(C)$ to $C_{d}$ and back via Hartogs' theorem.   

We now describe the divisor classes from Theorem A.  For a line bundle $\mathcal{L}$ on $C$ satisfying $e:=\dim|\mathcal{L}| \leq d \leq \deg{\mathcal{L}},$ the $e-$dimensional cycle $\Gamma_{d}(\mathcal{L})$ on $C_{d}$ parametrizes all effective divisors $D$ of degree $d$ that are subordinate to $|\mathcal{L}|,$ i.e. those $D$ for which $|\mathcal{L}(-D)| \neq \emptyset.$  A natural inner bound for the effective cone of $C_{g-1}$ in the fourth quarter of the $(\theta,x)-$plane is furnished by the cycle $\Gamma_{g-1}(K_{C}(-p))$ (where $p$ is a given point in $C$).  This is a divisor whose class is $\theta-x.$  

In Theorem 5 of \cite{Kou}, Kouvidakis obtains $\theta-2x$ as an outer bound for the effective cone of $C_{g-1}$ by degeneration to a hyperelliptic curve.  The inner bound $\theta-x$ reflects the fact that the canonical series $|K_{C}|$ separates 0-jets (i.e. is basepoint free) and Kouvidakis' outer bound is obtained by degeneration to the case in which $|K_{C}|$ fails to separate 1-jets (i.e. fails to be an immersion).  The divisor we obtain which spans a boundary ray of the effective cone of $C_{g-1}$ (in the nonhyperelliptic case) is supported on the set $$\bigcup_{p \in C}\Gamma_{g-1}(K_{C}(-2p))$$ and so it reflects in a precise manner the fact that $K_{C}$ separates 1-jets when $C$ is nonhyperelliptic.  Indeed, this divisor is the ramification locus of the Gauss map $\gamma: C_{g-1} \dashrightarrow (\mathbb{P}^{g-1})^{\ast}$ which assigns to a general $D \in C_{g-1}$ the hyperplane spanned by the image of $D$ under the canonical embedding.

More generally, our divisor on $C_{d}$ for $d \leq g-1$ is supported on the set $$\bigcup_{p \in C}\Gamma_{d}(K_{C}(-(g-d+1)p))$$  This is the pullback via $\widetilde{\tau}$ of the divisor on $G^{g-d-1}_{2g-2-d}(C)$ parametrizing linear series with ramification points of degree $g-d+1$ or higher.  

The bound for the effective cone implied by (i) of Theorem A is not sharp in the range $3 \leq d \leq \frac{g}{2}$; Theorem 5 of \cite{Kou} implies that the class $\theta-2x$ is effective for all such $d$.  Our next result gives a new bound for the effective cone of $C_{\frac{g+1}{2}}.$  Recall that $W^{r}_{d}(C)$ is the determinantal subvariety of $\textnormal{Pic}^{d}(C)$ parametrizing line bundles with at least $r+1$ global sections.

\begin{thmf}
Let $k \geq 3$ be an integer and let $C$ be a general curve of genus $2k-1.$  
\item[(i)]The class $\theta-(2-\frac{1}{k})x$ on $C_{k}$ is $\mathbb{Q}-$effective.
\item[(ii)] For all $t > 2-\frac{1}{k},$ the stable base locus of any divisor with class proportional to $\theta-tx$ contains the surface $$Z_{k}:=\bigcup_{\mathcal{L} \in W^{1}_{k+1}(C)}\Gamma_{k}(\mathcal{L})$$
\item[(iii)]The class $\theta-(2-\frac{1}{k})x$ spans a boundary ray of the effective cone of $C_{k}$ when $k=3.$  That is, the class $\theta-\frac{5}{3}x$ spans a boundary ray of the effective cone of $C_{3}$ when $C$ is a general curve of genus 5.
\end{thmf}

The divisor we consider is supported on the set $$\bigcup_{\mathcal{L} \in W^{1}_{k+1}(C)}\Gamma_{k}(K_{C}\otimes\mathcal{L}^{-1})$$  This is a direct generalization of the divisor shown by Pacienza in \cite{Pac} to span a boundary ray of the \textit{nef} cone of $C_{\frac{g}{2}}.$  
However, it is  \textit{not} nef in general, since in the case $k=3$ its top self-intersection is negative.  We plan to address the nef cone of $C_{\frac{g+1}{2}}$ further in future work.  

Special linear series on an arbitrary curve are poorly understood, and this accounts for much of the difficulty in the study of their symmetric powers.  However, special linear series on \textit{hyperelliptic} curves are understood quite well, and thus their consideration is a natural step in our study. 

\begin{thmg} 
\label{hyperell} 
Let $C$ be a hyperelliptic curve of genus $g$, and let $2 \leq d \leq g.$
\item[(i)] The class $\theta-(g-d+1)x$ spans a boundary ray of the effective cone of $C_{d}.$
\item[(ii)] The class $\theta$ spans a common boundary ray of the nef and movable cones of $C_{d}.$
\item[(iii)] For all $t \in [0,g-d+1],$ $$\textnormal{vol}_{C_{d}}(\theta-tx)=\frac{g!}{(g-d)!} \cdot \bigg(1-\frac{t}{g-d+1}\bigg)^{d}$$ 
\end{thmg}  

The proof of this result, which is given in Section \ref{hyp}, is based on the observation that if $C$ is hyperelliptic, then the Abel map $a_{d}:C_{d} \rightarrow \textrm{Pic}^{d}(C)$ is a divisorial contraction for $2 \leq d \leq g.$  (Note that the converse statement is also true by Martens' Theorem.)

Using a result of Pirola (Proposition \ref{pirola}) we are able to deduce that $N^{1}_{\mathbb{R}}(C_{d})$ is 2-dimensional when $C$ is a general hyperelliptic curve (Corollary \ref{n_s_rank}).  Combining (i) with Theorem 3 of \cite{Kou} then yields      

\begin{core1}
If $C$ is a general hyperelliptic curve of genus $g,$ then for $2 \leq d \leq g$ the effective cone of $C_{d}$ is spanned by the half-diagonal class $-\theta+(g+d-1)x$ and the class $\theta-(g-d+1)x.$
\end{core1}  

Since the nef and movable cone of $C_{d}$ share the ray spanned by $\theta$ as a common boundary, every big divisor class in the fourth quarter of the $(\theta,x)-$plane admits an "honest" Zariski decomposition which can be used to compute its volume; this is essentially the proof of (iii) of Theorem G.  Setting $t=1$ in both Theorem E and (iii) of Theorem G yields 
\begin{core2}
Let $C$ be a curve of genus $g \geq 4.$  Then
$$\textnormal{vol}_{C_{g-1}}(\theta-x)=\begin{cases} 
1&\textnormal{if $C$ is general nonhyperelliptic}\\
\frac{g!}{2^{g-1}}&\textnormal{if $C$ is general hyperelliptic}\end{cases}$$ 
\end{core2}
It is straightforward to check that when $C$ is nonhyperelliptic and $p_{1},...,p_{g-1}$ are general points on $C,$ the intersection of the divisors $\Gamma_{g-1}(K_{C}(-p_{1})), \dots \Gamma_{g-1}(K_{C}(-p_{g-1}))$ is the union of $C^{1}_{g-1}$ and the unique element of the linear system $|K_{C}(-p_{1} \dots -p_{g-1})|.$  As mentioned earlier, each of the divisors $\Gamma_{g-1}(K_{C}(-p_{i}))$ has numerical class $\theta-x.$  Consequently the formula $\textnormal{vol}_{C_{g-1}}(\theta-x)=1$ bears out the interpretation of the volume as "moving self-intersection." 

By a well-known elementary identity, $\frac{g!}{2^{g-1}}$ is an odd integer when $g$ is a power of $2$ and fails to be an integer otherwise, so that symmetric powers of hyperelliptic curves furnish a class of examples of integral divisors with non-integral volume.  Another class of such examples is treated in Section 2.3B of \cite{Laz1}.

We do not touch on the interesting issues concerning the effective cone of $C_{2};$ the curious reader is referred to \cite{Chan} and \cite{Ross}.    
 
\bigskip

\textbf{Acknowledgments:} The bulk of this paper is my Ph.D thesis at Stony Brook University.  I thank my advisor, Jason Starr, for his encouragement and support, as well as Lawrence Ein, Rob Lazarsfeld, Julius Ross, and Dror Varolin for valuable discussions.

I would also like to thank Li Li, Aleksey Zinger, and the anonymous referee for useful comments on the manuscript, Gianluca Pacienza for inviting me to IRMA to give a talk on this material, and Olivier Debarre, whose unpublished note has helped inspire the direction of this work.     

\bigskip

\textbf{Notation and Conventions:} We work over the field of complex numbers.  $C$ will always denote a smooth projective curve.  All cycle classes on smooth varieties lie in the algebraic cohomology ring with coefficients in $\mathbb{R}$.  If $\mathcal{Z}$ is a subvariety of the moduli space $\mathcal{M}_{g}$ of smooth projective curves of genus $g$, we say that a property holds for \textit{a general curve of $\mathcal{Z}$} if it holds on the complement of the union of countably many proper subvarieties of $\mathcal{Z}.$  If $\mathcal{Z}=\mathcal{M}_{g},$ we say that the property holds \textit{for a general curve of genus $g$}. 
\section{Preliminaries on $C_{d}$}
\subsection{The N\'{e}ron-Severi group of $C_{d}$}
Recall that the N\'{e}ron-Severi group $NS(X)$ of a smooth projective variety $X$ is the additive group of divisors on $X$ modulo algebraic equivalence.  

\medskip

\textbf{Definition:} The \textbf{(real) N\'{e}ron-Severi space} $N^{1}_{\mathbb{R}}(X)$ of $X$ is the real vector space $NS(X) \otimes_{\mathbb{Z}} \mathbb{R}.$

\medskip

The following result is well-known.
\begin{prop}
\label{ns_group}
For any $d \geq 2$ there is an isomorphism $N^{1}_{\mathbb{R}}(C_{d}) \simeq \mathbb{R} \oplus N^{1}_{\mathbb{R}}(J(C)).$ \hfill \qedsymbol
\end{prop}
Under this isomorphism, we may think of the summand $\mathbb{R}$ as being generated by $x$, and of the class $\theta$ as (not surprisingly) being contained in $N^{1}_{\mathbb{R}}(J(C)).$  In particular, $x$ and $\theta$ are linearly independent.

We may identify $NS(J(C))$ with the group $\textnormal{End}^{s}(J(C))$ of endomorphisms preserving the principal polarization (see Proposition 5.2.1 in \cite{BirLan} for details).  Since $NS(J(C))$ is torsion-free, we have that $N^{1}_{\mathbb{R}}(J(C))$ is 1-dimensional precisely when inversion is the only nontrivial automorphism in $\textnormal{End}^{s}(J(C)).$ 

The fact that this holds for a general Jacobian is due to Lefschetz.  The following refinement is due to Pirola.  (Recall that $\mathcal{J}_{g}$ is the Jacobian locus in the moduli space $\mathcal{A}_{g}$ of principally polarized abelian varieties of dimension $g.$)
\begin{prop}
\label{pirola}
\textnormal{((ii) of Proposition 3.4 in \cite{Pir})} Let $g \geq 2$ be an integer, and let $Y$ be a subvariety of codimension $\leq g-2$ in $\mathcal{J}_{g}.$  Then the rank of the N\'{e}ron-Severi group of an abelian variety corresponding to a general point of $Y$ is 1. \hfill \qedsymbol
\end{prop}

This result is sharp.  As pointed out in Remark 3.5 of \textit{loc. cit.}, its conclusion fails if we take $Y$ to be the $(2g-2)-$dimensional locus parametrizing Jacobians of bielliptic curves.

A curve $C$ is called \textit{e}-gonal if $e=\min\{k : C \textnormal{ admits a \textit{k}-to-1 covering of }\mathbb{P}^{1}\}.$
\begin{cor}
\label{n_s_rank}
If $C$ is a curve corresponding to a general point of the $e-$gonal locus of $\mathcal{M}_{g}$ (where $e \geq 2$) then $N^{1}_{\mathbb{R}}(C_{d})$ is 2-dimensional for all $d \geq 2.$  In particular, this is true of both the general curve of genus $g$ and the general hyperelliptic curve of genus $g.$
\end{cor}
\begin{proof}
Since the dimension of the $e-$gonal locus in $\mathcal{M}_{g}$ is $\min\{3g-3,2g+2e-5\},$ its image under the Torelli embedding is of codimension $\max\{0,g-2e+2\}$ in $\mathcal{J}_{g}.$  So the result follows immediately from Propositions \ref{ns_group} and \ref{pirola}.
\end{proof}
\subsection{Intersection theory on $C_{d}$}
The following formula, which is a consequence of the Poincar\'{e} formula (p.25 of \cite{ACGH}) will be used freely.  
\begin{lem}
\label{intersect}
For all $0 \leq k \leq d \leq g,$ $$x^{k}\theta^{d-k}=\displaystyle\frac{g!}{(g-d+k)!}$$
\end{lem}

\subsection{Subordinate Loci}
We now describe one of the most important constructions in this paper.  Let  $d \geq 2$ be an integer, let $(\mathcal{L},V)$ be a linear series of degree $n$ and dimension $r$ on $C,$ and assume that $n \geq d \geq r.$  Recall from \cite{ACGH} that there is a natural rank-$d$ vector bundle $E_{\mathcal{L}}$ on $C_{d}$ whose fibre over $D \in C_{d}$ is $H^{0}(D,\mathcal{L}|_{D}).$  As such, there is a morphism $$\alpha_{V}:V \otimes {\mathcal{O}}_{C_{d}} \rightarrow E_{\mathcal{L}}$$ whose fibre over each $D \in C_{d}$ is the restriction map $V \rightarrow H^{0}(D,\mathcal{L}|_{D}).$  The latter fails to be injective precisely when $D$ is subordinate to $(\mathcal{L},V),$ i.e. when $V \cap H^{0}(\mathcal{L}(-D)) \neq 0.$  

\medskip

\textbf{Definition:}  The cycle ${\Gamma}_{d}(\mathcal{L},V)$ is the degeneracy locus of $\alpha_{V}.$  If $V=H^{0}(\mathcal{L})$, we will write ${\Gamma}_{d}(\mathcal{L})$ instead.

\medskip

Note that $\Gamma_{d}(\mathcal{L},V)$ is supported on the set $$\{D \in C_{d}:V \cap H^{0}(\mathcal{L}(-D)) \neq 0\}.$$  

The following result computes the fundamental class of $\Gamma_{d}(\mathcal{L},V);$ we refer to p.342 of \cite{ACGH} for the proof.
\begin{lem}
\label{subordinate} (3.2 on p. 342 of \cite{ACGH}) Let $C$ be a curve of genus $g,$ and let $n,d,$ and $r$ be integers satisfying $n \geq d \geq r.$  Then ${\Gamma}_{d}(\mathcal{L},V)$ is $r-$dimensional, and its fundamental class is $$\displaystyle\sum_{k=0}^{d-r}\binom{n-g-r}{k}\frac{x^{k}\theta^{d-r-k}}{(d-r-k)!}.$$ 
\end{lem}
\subsection{Diagonal Calculations}  We collect here two special cases of the computation of diagonal classes in Proposition 5.1 on p.358 of \cite{ACGH} which are used in the proof of Theorem A.  First, we define the diagonal loci.

\medskip

\textbf{Definition:} Let $d \geq 2$ be an integer and let $a_{1}, \dots ,a_{k}$ be a sequence of positive integers which is a partition of $d.$  Then $\Delta_{a_{1}, \dots ,a_{k}}$ is the reduced subscheme of $C_{d}$ supported on the image of the morphism $$\phi_{a_{1}, \dots ,a_{k}}:C^{k} \rightarrow C_{d} \hspace{0.3cm} (p_{1}, \dots ,p_{k}) \mapsto \sum_{i=0}^{k}a_{i}p_{i}$$  $\Delta_{2,1, \dots 1}$ will be denoted by $\Delta.$ 

\medskip
  
\begin{prop}
\label{small_diag}
Let $C$ be a curve of genus $g.$  The fundamental class of ${\Delta}_{d}$ in $C_{d}$ is $$dx^{d-2} \cdot \Bigl(((d-1)g+1)x-(d-1)\theta\Bigr) $$
\end{prop}
\begin{proof}
By Proposition 5.1 on p.358 of \cite{ACGH}, this class is $${\sum}_{0 \leq \beta \leq \alpha \leq d-1}\frac{(-1)^{\alpha + \beta}}{{\beta}!(\alpha - \beta)!}\Bigl(d(\beta + 1 - g)+d^{2}(g - \beta)\Bigr)x^{d-1-{\alpha}}{\theta}^{\alpha}.$$  The result is thus immediate when $d=2;$ when $d \geq 3,$ it follows from the fact that ${\sum}_{1 \leq \beta \leq \alpha}(-1)^{\beta}{\beta}\binom{\alpha}{\beta}=0$ for all $\alpha\geq 2.$
\end{proof}
\begin{prop}
\label{other_diag}
The numerical class of $\Delta_{g-d+1,d}$ is $$(1-\frac{1}{2}\delta_{(d,\frac{g+1}{2})}) \cdot d(g-d+1)x^{g-3} \cdot \biggl\{\Bigl(d(g-d+1)(g^{2}-g)-(g^{2}+1)(g-2)\Bigr)x^{2}$$ $$+\Bigl((2-2d)g^{2}+(2d^{2}-3)g-(2d^{2}-d-2)\Bigr)x\theta+(d-1)(g-d)\theta^{2}\biggr\}$$ where $\delta_{(d,\frac{g+1}{2})}$ is the Kronecker delta.
\end{prop}
\begin{proof}
As a set, $\Delta_{g-d+1,d}$ is the image of the morphism $$\phi_{g-d+1,d}:C \times C \rightarrow C_{g+1}, (p,q) \mapsto (g-d+1)p+dq.$$  When $d \neq \frac{g+1}{2},$ this morphism is injective, so that the class of $\Delta_{g-d+1,d}$ is the pushforward class $[\phi_{g-d+1,d}]_{\ast}(C \times C).$  When $d=\frac{g+1}{2},$ we have the factorization
$$\xymatrix{
C \times C
\ar[rr]^{\phi_{\frac{g+1}{2},\frac{g+1}{2}}}
\ar[dr]_{\pi}
&& C_{g+1}\\
& C_{2} 
\ar[ur]_{\phi_{\frac{g+1}{2}}}}\\$$ where $\pi:C \times C \rightarrow C_{2}$ is the canonical quotient map and $\phi_{\frac{g+1}{2}}:C_{2} \rightarrow C_{g+1}$ is defined by $p+q \mapsto (\frac{g+1}{2})p+(\frac{g+1}{2})q.$  Since $\phi_{\frac{g+1}{2}}$ is injective, the class of $\Delta_{\frac{g+1}{2},\frac{g+1}{2}}$ is the pushforward class $[\phi_{\frac{g+1}{2}}]_{\ast}(C_{2}),$ which by our factorization is equal to $\frac{1}{2} \cdot [\phi_{\frac{g+1}{2},\frac{g+1}{2}}]_{\ast}(C \times C).$  Therefore the class of $\Delta_{g-d+1,d}$ is equal to $$(1-\frac{1}{2}\delta_{(d,\frac{g+1}{2})}) \cdot [\phi_{g-d+1,d}]_{\ast}(C \times C).$$
  
By Proposition 5.1 on p.358 of \cite{ACGH}, the class $[\phi_{g-d+1,d}]_{\ast}(C \times C)$ is the coefficient of $t_{1}t_{2}$ in the expression  $$\displaystyle\sum_{0 \leq \beta \leq \alpha \leq g-1}\frac{(-1)^{\alpha+\beta}}{\beta!(\alpha-\beta)!}\Bigl(1+(g-d+1)t_{1}+dt_{2}\Bigr)^{2-g+\beta}\Bigl(1+(g-d+1)^{2}t_{1}+d^{2}t_{2}\Bigr)^{g-\beta}x^{g-1-\alpha}\theta^{\alpha}.$$  This coefficient is equal to $$d(g-d+1)\displaystyle\sum_{\alpha=0}^{g-1}\frac{(-1)^{\alpha}}{\alpha!}\biggl\{\Bigl((d-1)g-d(d-1)\Bigr)\displaystyle\sum_{\beta=0}^{\alpha}(-1)^{\beta}\binom{\alpha}{\beta}\beta^{2}$$ $$+\Bigl((2-2d)g^{2}+(2d^{2}-d-2)g-(d^{2}-d-1)\Bigr)\displaystyle\sum_{\beta=0}^{\alpha}(-1)^{\beta}\binom{\alpha}{\beta}\beta$$ $$+\Bigl((d-1)g^{3}-(d^{2}-2)g^{2}+(d^{2}-d-1)g+2\Bigr)\displaystyle\sum_{\beta=0}^{\alpha}(-1)^{\beta}\binom{\alpha}{\beta}\biggr\}x^{g-1-\alpha}\theta^{\alpha}$$  Since the three sums over $\beta$ are equal to 0 for all $\alpha \geq 3,$ we finally obtain that the class $[\phi_{g-d+1,d}]_{\ast}(C \times C)$ is equal to $$d(g-d+1)x^{g-3}\biggl\{\Bigl((d-1)g^{3}-(d^{2}-2)g^{2}+(d^{2}-d-1)g+2\Bigr)x^{2}$$ $$+\Bigl((2-2d)g^{2}+(2d^{2}-3)g-(2d^{2}-2d-1)\Bigr)x\theta+\Bigl((d-1)g-d(d-1)\Bigr)\theta^{2}\biggr\}.$$
\end{proof}
\section{Results from the asymptotic theory of linear series}
\label{asymp_thy}
We collect the results on stable base loci and the volume function that will be used in the sequel.  We refer to \cite{ELMNP} and \cite{Laz1} for a thorough treatment.    

\medskip

\subsection{Base Loci}
\label{baseloc}
Recall the following definition from the Introduction.

\medskip

\textbf{Definition:} Let $X$ be an irreducible projective variety and let $D$ be a $\mathbb{Q}-$Cartier divisor on $X.$  Then the \textbf{stable base locus of} $D$ is the algebraic set $$\textbf{B}(D)=\bigcap_{m}\textnormal{Bs}|mD|$$ where the intersection is taken over all positive integers $m$ for which $mD$ is Cartier.  

\medskip

For the proof of Theorem \ref{stab_classes} we will need to know that the stable base locus of a Cartier divisor $D$ can be realized as the base locus of some multiple of $D.$  The relevant result, which we state below, follows immediately from Proposition 2.1.21 in \cite{Laz1}.

\begin{prop}
\label{multiple}
Let $D$ be a Cartier divisor on an irreducible projective variety $X.$
\begin{itemize}
\item[(i)]{There exists a positive integer $m_{0}$ such that $\textnormal{\textbf{B}}(D)=\textnormal{Bs}|km_{0}D|$ for all $k >> 0.$}
\item[(ii)]{$\textnormal{\textbf{B}}(D)=\textnormal{\textbf{B}}(mD)$ for all $m \geq 1.$}
\end{itemize} 
In particular, we can always find a positive integer $m$ for which $\textnormal{\textbf{B}}(mD)=\textnormal{Bs}|mD|.$ \hfill \qedsymbol
\end{prop}

As mentioned in the introduction, the stable base locus of a $\mathbb{Q}-$Cartier divisor is not a numerical invariant; this can be seen by considering any smooth projective $X$ with $H^{1}(X,\mathcal{O}_{X}) \neq 0$ and comparing the trivial line bundle on $X$ to a non-torsion line bundle of degree 0 on $X.$  We now introduce the "outer and inner approximations" of the stable base locus.

\medskip

\textbf{Definition:} Let $X$ be an irreducible projective variety and let $D$ be an $\mathbb{R}-$Cartier divisor on $X.$
\begin{itemize}
\item[(i)]{The \textbf{augmented base locus} of $D$ is $$\textbf{B}_{+}(D)=\bigcap_{A}\textbf{B}(D-A)$$ where the intersection is taken over all ample $\mathbb{R}-$Cartier divisors $A$ for which $D-A$ is $\mathbb{Q}-$Cartier.}
\item[(ii)]{The \textbf{restricted base locus} of $D$ is $$\textbf{B}_{-}(D)=\bigcup_{A}\textbf{B}(D+A)$$ where the union is taken over all ample $\mathbb{R}-$Cartier divisors $A$ for which $D+A$ is $\mathbb{Q}-$Cartier.}
\end{itemize}

\bigskip

While $\textbf{B}_{+}(D)$ is known to be Zariski-closed, we only know at this point that $\textbf{B}_{-}(D)$ is at worst a countable union of subvarieties of $X$ (Proposition 1.19 in \cite{ELMNP}).

The following result is a straightforward consequence of the definitions.
\begin{lem} 
\label{basicloc}
For all $\mathbb{R}-$Cartier divisors $D$ on $X,$ the following statements hold:
\begin{itemize}
\item[(i)]{$\textnormal{\textbf{B}}_{-}(D)$ and $\textnormal{\textbf{B}}_{+}(D)$ are numerical invariants of $D,$ so that they are both well-defined for any class $D \in N^{1}_{\mathbb{R}}(X).$}
\item[(ii)]{If $D$ is a $\mathbb{Q}-$Cartier divisor, then $\textnormal{\textbf{B}}_{-}(D) \subseteq \textnormal{\textbf{B}}(D) \subseteq \textnormal{\textbf{B}}_{+}(D).$ \hfill \qedsymbol} 
\end{itemize} 
\end{lem}

Ampleness, nefness, and bigness can all be characterized in terms of augmented and restricted base loci: 
\begin{lem} 
\label{char_cones}
For all $D \in N^{1}_{\mathbb{R}}(X),$ the following hold:
\begin{itemize}
\item[(i)] $D$ is ample if and only if $\textnormal{\textbf{B}}_{+}(D)=\emptyset.$
\item[(ii)] $D$ is nef if and only if $\textnormal{\textbf{B}}_{-}(D)=\emptyset.$
\item[(iii)] $D$ is big if and only if $\textnormal{\textbf{B}}_{+}(D) \neq X.$
\item[(iv)] $D$ is pseudoeffective if and only if $\textnormal{\textbf{B}}_{-}(D) \neq X.$ \hfill \qedsymbol
\end{itemize}
\end{lem}

We refer to \cite{ELMNP} for the proofs.  In \textit{loc. cit.},(i) and (iii) are Example 1.7, while (ii) and (iv) are Example 1.18. 

\medskip

\textbf{Definition:} A class $D \in N^{1}_{\mathbb{R}}(X)$ is \textbf{stable} if $\textbf{B}_{-}(D)=\textbf{B}_{+}(D).$

\medskip

Note that if $D$ is stable, we may speak of its stable base locus $\textbf{B}(D),$ and that the stable classes with empty stable base locus are precisely the ample classes.

The next two results will be used in the proofs of Theorems C and D; we refer to \cite{ELMNP} for their proofs.
\begin{prop}
\label{stab_dense}
(1.26 in \cite{ELMNP}) The set of stable classes is open and dense in $N^{1}_{\mathbb{R}}(X).$  In fact, for every $D \in N^{1}_{\mathbb{R}}(X)$ there exists $\epsilon > 0$ such that for any ample class $A$ satisfying $\lVert A \rVert < \epsilon,$ $D-A$ is stable. \hfill \qedsymbol  
\end{prop}

\begin{prop}
\label{stab_drop}
(1.21 in \cite{ELMNP}) For every $\mathbb{R}-$divisor $D$, there is an $\epsilon > 0$ such that $\textnormal{\textbf{B}}_{-}(D-A)=\textnormal{\textbf{B}}_{+}(D-A)=\textnormal{\textbf{B}}_{+}(D)$ for every ample $A$ with $\lVert A \rVert < \epsilon.$ \hfill \qedsymbol
\end{prop}

The following theorem of Nakamaye gives a useful characterization of the augmented base locus of a nef and big divisor.  We refer to \cite{Nak} or p.249-251 in \cite{Laz2} for the proof.
\begin{thm}
\label{nak_base_loc}
(Theorem 0.3 in \cite{Nak})  If $D$ is a nef and big divisor on a smooth projective variety $X,$ then $\textbf{B}_{+}(D)$ is the union of all positive-dimensional subvarieties $V$ of $X$ for which $D^{\dim{V}} \cdot V = 0.$ \hfill \qedsymbol
\end{thm}

\medskip

\textbf{Remark:}  This result, while entirely sufficient for our purposes, has been generalized to arbitrary big divisors; see Theorem C in \cite{ELMNP2}.

\subsection{The Volume Function}  
First, we briefly outline how the definition of the volume given in the introduction gives rise to a continuous function on the N\'{e}ron-Severi space.  In this subsection, $X$ will always denote an irreducible projective variety.

We omit proofs, referring to Section 2.2C of \cite{Laz1}.
\begin{prop}
\textnormal{((i) of Proposition 2.2.35 in \cite{Laz1})} If $k$ is a positive integer and $D$ is a Cartier divisor on $X,$ then $$\textnormal{vol}_{X}(kD)=k^{n} \cdot \textnormal{vol}_{X}(D).$$ 
\end{prop}
\hspace{0.6cm}As a result, for any $\mathbb{Q}-$Cartier divisor $D$ on $X,$ we may define $$\textnormal{vol}_{X}(D):=\frac{1}{k^{n}} \cdot \textnormal{vol}_{X}(kD),$$ where $k$ is a positive integer for which $mD$ is an integral Cartier divisor.
\begin{prop}
\label{vol_num_inv}   
\textnormal{(Proposition 2.2.41 in \cite{Laz1})} Let $D_{1}$ and $D_{2}$ be Cartier divisors on $X$ which are numerically equivalent.  Then $$\textnormal{vol}_{X}(D_{1})=\textnormal{vol}_{X}(D_{2}).$$
\end{prop}
\hspace{0.6cm}It follows that $\textrm{vol}_{X}:N^{1}_{\mathbb{Q}}(X) \rightarrow [0,\infty)$ is a well-defined function.  The next result ensures that $\textrm{vol}_{X}$ can be uniquely extended to a continuous real-valued function on $N^{1}_{\mathbb{R}}(X).$
\begin{thm}
\label{vol_cont}
\textnormal{(Theorem 2.2.44 in \cite{Laz1})} Let $\lVert \cdot \rVert$ be any norm on $N^{1}_{\mathbb{R}}(X).$  Then there exists a constant $C > 0$ such that $$|\textnormal{vol}_{X}(\eta)-\textnormal{vol}_{X}({\eta}')| \leq C \cdot (\max(\lVert \eta \rVert , \lVert {\eta}' \rVert))^{n-1} \cdot \lVert{\eta-{\eta}'}\rVert$$ for any two classes $\eta,{\eta}' \in N^{1}_{\mathbb{Q}}(X).$
\end{thm}
As stated earlier, the volume of any ample divisor is its top self-intersection.  The following generalization, which is an immediate corollary of Theorem 1.4.40 in \cite{Laz1}, will be used in the proof of Theorems E and G.
\begin{lem}
\label{vol_nef}
If $D$ is a nef divisor on an irreducible projective variety $X,$ then $\textnormal{vol}_{X}(D)=D^{n}.$
\end{lem}
\section{The Residuation Map}
\label{resid}
We begin with the natural generalization of the Abel map (which was defined in the introduction).  

Let $C$ be a curve, and let $r$ and $d$ be positive integers.  Then there is a fine moduli variety $G^{r}_{d}(C)$ parametrizing linear series of degree $d$ and dimension $r$ on $C,$ a determinantal subvariety $W^{r}_{d}(C)$ of $\textnormal{Pic}^{d}(C)$ parametrizing line bundles on $C$ of degree $d$ with at least $r+1$ global sections  (we will refer to $W^{0}_{d}(C)$ as $W_{d}(C)$), and a natural morphism $$a_{r,d}:G^{r}_{d}(C) \rightarrow W^{r}_{d}(C)$$ which takes each linear series to its associated line bundle (we refer to Section 3 of Chapter IV of \cite{ACGH} for details).    

The exceptional locus of $a_{r,d}$ is the subvariety $\mathfrak{I}^{r}_{d}(C)$ of $G^{r}_{d}(C)$ which parametrizes incomplete linear series.  

Since $G^{0}_{d}(C)$ is canonically isomorphic to $C_{d},$ we have that $$a_{0,d}:C_{d} \rightarrow W_{d}(C)$$ is the Abel map $a_{d}.$  The inverse image $a_{0,d}^{-1}(W^{r}_{d}(C)),$ which is denoted by $C^{r}_{d},$ parametrizes effective divisors of degree $d$ which move in a linear series of dimension at least $r.$  Note that the exceptional locus $\mathfrak{I}^{0}_{d}(C)$ of $a_{0,d}$ is simply $C^{1}_{d}.$  

If $C$ is a curve of genus $g$ and $2 \leq d \leq g-1,$ then there is an isomorphism $$\xymatrix{\tau : \textnormal{Pic}^{2g-2-d}(C) \ar[r]^-{\simeq} &\textnormal{Pic}^{d}(C)\\}$$ defined by $\tau(\mathcal{L})=K_{C} \otimes \mathcal{L}^{-1}.$  By the Riemann-Roch theorem, this restricts to an isomorphism $$\xymatrix{\tau : W^{g-d-1}_{2g-2-d}(C) \ar[r]^-{\simeq} &W_{d}(C)\\}.$$  This in turn lifts via $a_{0,d}$ and $a_{g-d-1,2g-2-d}$ to a birational map $$\widetilde{\tau} : G^{g-d-1}_{2g-2-d}(C) \dashrightarrow C_{d}.$$  Note that $\widetilde{\tau}$ is a biregular isomorphism precisely when $C^{1}_{d}=\emptyset.$  

The following theorem is due to Gieseker.  
\begin{thm} 
\label{gieseker}
(1.6 on p. 214 of \cite{ACGH}) Let $C$ be a general curve of genus $g.$  Let $d,r$ be integers satisfying $d \geq 1$ and $r \geq 0.$  Then $G^{r}_{d}(C)$ is smooth of dimension $g-(r+1)(g-d+r).$ 
\end{thm}
It is a straightforward consequence of Theorem \ref{gieseker} that when $C$ is a general curve of genus $g$ the codimension of $C^{1}_{d}$ is $g-d+1$ and the codimension of $\mathfrak{I}^{g-d-1}_{2g-2-d}(C)$ is 2, so that $\widetilde{\tau}$ is an isomorphism of smooth varieties in codimension 1.    

Another consequence is that a general curve of genus $g$ is $\lceil{\frac{g}{2}+1}\rceil-$gonal.  For all such curves, then, $\widetilde{\tau}$ fails to be a biregular isomorphism whenever $\frac{g}{2}+1 \leq d \leq g-1.$    
\begin{prop}
\label{picard_isom}
Let $C$ be a general curve of genus $g \geq 3$ and let $d \geq 3$ be an integer satisfying $\frac{g}{2}+1 \leq d \leq g-1$.      
\begin{itemize}
\item[(i)]{$\widetilde{\tau}$ induces an isomorphism $$\xymatrix{
\widetilde{\tau}^{\ast} : \textnormal{Pic}(C_{d}) \ar[r]^-{\simeq} &\textnormal{Pic}(G^{g-d-1}_{2g-2-d}(C))\\}.$$}
\item[(ii)]{Let $\mathcal{L}$ be a line bundle on $C_{d}$ and $\mathcal{M}$ be a line bundle on $G^{g-d-1}_{2g-2-d}(C).$  Then 
\begin{align*}
H^{0}(C_{d},(\tau^{-1})^{\ast}\mathcal{M}) &\simeq H^{0}(G^{g-d-1}_{2g-2-d}(C),\mathcal{M})\\
H^{0}(G^{g-d-1}_{2g-2-d}(C),\tau^{\ast}\mathcal{L}) &\simeq H^{0}(C_{d},\mathcal{L})
\end{align*}}
\end{itemize}
\end{prop}

\begin{proof}
Let $U=C_{d}-C^{1}_{d}$ and $V=G^{g-d-1}_{2g-2-d}(C)-\mathfrak{I}^{g-d-1}_{2g-2-d}(C).$  Clearly ${\tau}|_{V}:V \rightarrow U$ is an isomorphism and ${\tau}^{-1}|_{U}:U \rightarrow V$ is its inverse.  These furnish natural isomorphisms $$H^{0}(U,\mathcal{L}|_{U})\simeq H^{0}(V,{\tau}^{\ast}\mathcal{L}|_{V})$$ $$H^{0}(V,\mathcal{M}|_{V}) \simeq H^{0}(U,({\tau}^{-1})^{\ast}\mathcal{M}|_{U})$$  

Since $C$ is general, it follows from Hartogs' theorem that these isomorphisms extend to all of $C_{d}$ and $G^{g-d-1}_{2g-2-d}(C).$
\end{proof}

\textbf{Remark:}  It follows from the first part of this proof that for any curve $C$ the singular locus of $G^{g-d-1}_{2g-2-d}(C)$ is contained in $\mathfrak{I}^{g-d-1}_{2g-2-d}(C).$

\medskip

\textbf{Remark:}  It is possible for $\mathfrak{I}^{g-d-1}_{2g-2-d}(C)$ to be of codimension 1 in $G^{g-d-1}_{2g-2-d}(C)$ even if $C^{1}_{d}$ is of codimension 2 in $C_{d}.$  For example, if $C$ is a nonhyperelliptic trigonal curve of genus 5, then $C^{1}_{3} \simeq \mathbb{P}^{1}$ and $\mathfrak{I}^{1}_{5}(C) \simeq \mathbb{P}^{2}.$

\medskip

The following corollary is immediate.

\begin{cor}
$\widetilde{\tau}$ induces an isomorphism $$\xymatrix{
\widetilde{\tau}^{\ast} : N^{1}_{\mathbb{R}}(C_{d}) \ar[r]^-{\simeq} &N^{1}_{\mathbb{R}}(G^{g-d-1}_{2g-2-d}(C))\\}$$ of N\'{e}ron-Severi spaces. \hfill \qedsymbol
\end{cor}

Our next task is to determine $\widetilde{\tau}^{\ast}$ explicitly.  The divisor classes $x$ and $\theta$ on $C_{d}$ can be generalized in a straightforward fashion to obtain divisor classes on $G^{g-d-1}_{2g-2-d}(C).$ 

Recall that a vector bundle $E$ is ample if the hyperplane class on the \textit{subbundle} projectivization $\mathbb{P}_{sub}(E^{\ast})$ of $E^{\ast}$ is ample.

\begin{lem}
\label{det_ample}
Let $X$ be a smooth projective variety, and let $E$ be an ample vector bundle of rank $s$ on $X.$  Then for all $s' \leq s,$ the Pl\"{u}cker class on the associated Grassmann bundle $G(s',E^{\ast})$ of rank-$s'$ subbundles of $E^{\ast}$ is ample.
\end{lem}
\begin{proof}
If $\nu:G(s',E^{\ast}) \rightarrow X$ is the structure map, then the determinant of the inclusion $S_{s',E^{\ast}} \hookrightarrow \nu^{\ast}(E)$ of the tautological subbundle induces the Pl\"{u}cker embedding $G(s',E^{\ast}) \hookrightarrow \mathbb{P}_{sub}({\wedge}^{s'}E^{\ast}).$  The result then follows from the fact that the amplitude of $E$ implies the amplitude of its exterior powers (part (ii) of Corollary 6.1.16 on p.15 of \cite{Laz2}).  
\end{proof}

While we feel that the following result should be well known, we include its proof for lack of a reference.

\begin{prop}
\label{plucker_ample}
For any effective divisor $D$ of degree $r+1$ on $C,$ the set $$\widehat{X}_{D}:=\{(V,\mathcal{M}) \in G^{r}_{d}(C):V \cap H^{0}(\mathcal{M}(-D)) \neq 0\}$$ has the natural structure of an ample divisor on $G^{r}_{d}(C).$
\end{prop}

\begin{proof}
We first recall some aspects of the construction of $G^{r}_{d}(C)$ in Section 3 of Chapter IV of \cite{ACGH}.  Fix a Poincar\'{e} bundle $\mathcal{L}$ on $C \times \textrm{Pic}^{d}(C),$ and let $D'$ be an effective divisor on $C$ of degree $2g-d-1.$  

If $\eta:C \times \textrm{Pic}^{d}(C) \rightarrow \textrm{Pic}^{d}(C)$ is projection onto the second factor, and $\Gamma$ is the product divisor $(D+D') \times \textrm{Pic}^{d}(C),$ then the direct image sheaf $\eta_{\ast}\mathcal{L}(\Gamma)$ is locally free of rank $g+r+1$ and its fibre over a line bundle $\mathcal{M}$ of degree $d$ on $C$ is $H^{0}(\mathcal{M}(D+D')).$  Indeed, $h^{1}(\mathcal{M}(D+D'))=0$ by Serre duality since the degree of $\mathcal{M}(D+D')$ is $2g+r$, so this follows from Riemann-Roch and base change in cohomology.   

If ${\Gamma}'$ is the product divisor $D' \times \textrm{Pic}^{d}(C),$ then an entirely analogous argument tells us that $\eta_{\ast}\mathcal{L}({\Gamma}')$ is a rank-$g$ subbundle of $\eta_{\ast}\mathcal{L}(\Gamma).$  Since the dual of $\eta_{\ast}\mathcal{L}(\Gamma)$ is ample by Proposition 2.2 on p.310 of \cite{ACGH}, the Pl\"{u}cker divisor ${\sigma}'$ on the Grassmann bundle $G(r+1,\eta_{\ast}\mathcal{L}(\Gamma))$ associated to $\eta_{\ast}\mathcal{L}({\Gamma}')$ is ample by Lemma \ref{det_ample}. 

For each line bundle $\mathcal{M}$ of degree $d$ on $C,$ there is a commutative diagram
$$\xymatrix{
0 \ar[r] &H^{0}(\mathcal{M}) \ar[r]^{f_{1}} &H^{0}(\mathcal{M}(D+D')) \ar[r]^{g_{1}} &H^{0}(\mathcal{M}(D+D')|_{D+D'})\\
0 \ar[r] &H^{0}(\mathcal{M}(-D)) \ar[u]^{i} \ar[r]^{f_{2}} &H^{0}(\mathcal{M}(D')) \ar[u]^{i'} \ar[r]^{g_{2}} &H^{0}(\mathcal{M}(D')|_{D+D'}) \ar[u]^{i''}}$$
in which both rows are exact and all vertical arrows are injective, so that a diagram chase gives the equality $$(f_{1} \circ i)\bigl(H^{0}(\mathcal{M}(-D))\bigr)=f_{1}\bigl(H^{0}(\mathcal{M})\bigr) \cap i'\bigl(H^{0}(\mathcal{M}(D'))\bigr).$$  

Therefore if $V$ is a subspace of $H^{0}(\mathcal{M}),$ we have $$f_{1}(V) \cap (f_{1} \circ i)\bigl(H^{0}(\mathcal{M}(-D))\bigr)=f_{1}(V) \cap i'\bigl(H^{0}(\mathcal{M}(D'))\bigr).$$  

It follows at once from the definitions that $\widehat{X}_{D}=G^{r}_{d}(C) \cap {\sigma}'.$  Consequently $\widehat{X}_{D},$ which has a cycle structure induced by its being an intersection of cycles and is the restriction of an ample divisor, is ample.
\end{proof} 

As $D$ varies over $C_{r+1},$ the divisors $\widehat{X}_{D}$ sweep out an algebraic family whose common numerical class we will denote by $\widehat{X}.$  Also, we will denote by $\widehat{\theta}$ the numerical class of the pullback to $G^{r}_{d}(C)$ of a theta-divisor on $\textrm{Pic}^{d}(C).$    

\begin{prop}
\label{explicit_isom}
Under the isomorphism ${\widetilde{\tau}}^{\ast}:N^{1}_{\mathbb{R}}(C_{d}) \rightarrow N^{1}_{\mathbb{R}}(G^{g-d-1}_{2g-2-d}(C)),$ $${\widetilde{\tau}}^{\ast}(\theta)=\widehat{\theta}, \hspace{0.2cm} {\widetilde{\tau}}^{\ast}(\theta-x)=\widehat{X}.$$
\end{prop}
\begin{proof}
If $\widehat{\tau}: \textrm{Pic}^{2g-2-d}(C) \rightarrow \textrm{Pic}^{d}(C)$ is the morphism induced by taking Serre duals and $\mathcal{L}_{1}$ and $\mathcal{L}_{2}$ are line bundles on $C$ of respective degrees $2g-2-d$ and $d$ satisfying $\mathcal{L}_{1} \otimes \mathcal{L}_{2} \simeq K_{C},$ we have the commutative diagram
$$\xymatrix{
\textrm{Pic}^{0}(C) \ar[d]^{(-1)} \ar[r]^{\cdot \otimes \mathcal{L}_{1}} &\textrm{Pic}^{2g-2-d}(C) \ar[d]^{\widehat{\tau}} &G^{g-d-1}_{2g-2-d}(C) \ar[l] \ar@{-->}[d]^{\tau}\\
\textrm{Pic}^{0}(C) \ar[r]^{\cdot \otimes \mathcal{L}_{2}} &\textrm{Pic}^{d}(C) &C_{d} \ar[l]}$$ where the leftmost vertical arrow is multiplication by $-1$ and the left horizontal arrows on the top and bottom are multiplication by $\mathcal{L}_{1}$ and $\mathcal{L}_{2},$ respectively.  Since multiplication by $-1$ induces the identity on the cohomology of $\textrm{Pic}^{0}(C)$, it induces the identity on the N\'{e}ron-Severi group of $\textrm{Pic}^{0}(C).$  Therefore ${\tau}^{\ast}$ takes the theta class on $C_{d}$ to $\widehat{\theta}$ on $G^{g-d-1}_{2g-2-d}(C).$ 

If $D$ is an effective divisor of degree $g-d$ on $C,$ then it follows immediately from Riemann-Roch that $D' \in C_{d}$ satisfying $h^{0}(D')=1$ is subordinate to $|K_{C}(-D)|$ precisely when $D$ is subordinate to $|K_{C}(-D')|.$  This can be rephrased as saying that away from the loci of indeterminacy of $\tau$ and ${\tau}^{-1},$ the divisor $\widehat{X}_{D}$ is isomorphic to ${\Gamma}_{d}(K_{C}(-D))$ via $\tau.$  Since the fundamental class of ${\Gamma}_{d}(K_{C}(-D))$ is $\theta-x$ by Lemma \ref{subordinate}, we have that $\tau^{\ast}(\theta-x)=\widehat{X}.$
\end{proof}   

Since $\widetilde{\tau}$ is a biregular isomorphism if $C^{1}_{d}=\emptyset$, an immediate application of Proposition \ref{explicit_isom} yields the following result.  Note that $C$ is not assumed to be general.

\begin{cor}
\label{subnef}
If $C$ is any curve and $d$ is a positive integer for which $C^{1}_{d}=\emptyset,$ then the class $\theta-x$ on $C_{d}$ is ample. \hfill \qedsymbol
\end{cor}

In the case where $C$ is a general curve of even genus $g,$ Corollary \ref{subnef} is subsumed by Pacienza's computation of the nef cone of $C_{\frac{g}{2}}$ in \cite{Pac}.  However, it yields the best inner bound for the nef cone of $C_{d}$ for $3 \leq d \leq \frac{g+1}{2}$ currently available when $C$ is a general curve of odd genus $g \geq 5.$

\section{Proofs of The Main Results}
\label{proofs}  
\subsection{The Nonhyperelliptic Case} In this section we prove Theorems A,C,D,E, and F.
\label{nonhyp}

\subsubsection{The Effective Cone}
\label{seceff}
The following calculation establishes (i) of Theorem A.
\begin{prop}
Let $C$ be a very general curve of genus $g \geq 4.$  For $2 \leq d \leq g$ let $D_{d}$ be the reduced divisor on $C_{d}$ supported on the set $$\bigcup_{p \in C}{\Gamma}_{d}\bigl(K_{C}(-(g-d+1)p)\bigr)$$  Then the numerical class of $D_{d}$ is $$(g-d+1)\bigl((g^{2}-dg+(d-2))\theta-(g^{2}-(d-1)g-2)x\bigr).$$
\end{prop}
\begin{proof}
Fix $d-1$ general points $q_{1},...q_{d-1}$ on C.  The two test curves in $C_{d}$ that we will use to compute the numerical class of $D_{d}$ are $\Delta_{d}$ and the curve $${\chi}_{d}:=\bigcap_{j=1}^{d-1}X_{q_{j}}=\{p+q_{1}+...+q_{d-1}:p \in C\}$$  The numerical class of ${\chi}_{d}$ is $x^{d-1},$ and by Proposition \ref{other_diag}, the numerical class of $\Delta_{d}$ is $dx^{d-2}(((d-1)g+1)x-(d-1)\theta)$.  

\hspace{0.6cm}The intersection number ${\chi}_{d} \cdot D_{d}$ is the cardinality of the set $$\biggl\{q \in C : \exists p \in C \ni (g-d+1)p+q \leq |K_{C}(-q_{1}-...-q_{d-1})| \biggr\}$$  If $(g-d+1)p \leq |K_{C}(-q_{1}-...-q_{d-1})|,$ then there are $$((2g-2)-(d-1))-(g-d+1)=g-2$$ points $q$ (counting multiplicity) such that $$(g-d+1)p+q \leq |K_{C}(-q_{1}-...-q_{d-1})|,$$ and we may conclude that $${\chi}_{d} \cdot D_{d}=(g-2) \cdot {\Delta}_{g-d+1} \cdot {\Gamma}_{g-d+1}(K_{C}(-q_{1}-...-q_{d-1})).$$  

\hspace{0.6cm}When $d \neq \frac{g+1}{2},$ the intersection number $\Delta_{d} \cdot D_{d}$ is the cardinality of the set $$\mathfrak{K}_{g,d}:=\biggl\{q \in C: \exists p \in C \ni (g-d+1)p + dq \leq |K_{C}| \biggr\}$$ and when $d=\frac{g+1}{2},$ we have that $\Delta_{d} \cdot D_{d}$ is twice the cardinality of $\mathfrak{K}_{g,d}.$  Therefore our system of equations is 
\begin{align*}
{\chi}_{d} \cdot D_{d} &= (g-2) \cdot {\Delta}_{g-d+1} \cdot {\Gamma}_{g-d+1}(K_{C}(-q_{1}-...-q_{d-1}))\\  
\Delta_{d} \cdot D_{d} &= (1+\delta_{(d,\frac{g+1}{2})}) \cdot \Delta_{g-d+1,d} \cdot \Gamma_{g+1}(K_{C})
\end{align*}
\hspace{0.6cm}If the numerical class of $D_{d}$ is $a\theta-bx$ for $a,b \in \mathbb{R},$ this becomes
\begin{align*}
ag-b &= g^{4}-2dg^{3}+(d^{2}+2d-4)g^{2}-(2d^{2}-5d+1)g-(2d-2)\\ 
adg-b &= dg^{4}-(2d^{2}-d+1)g^{3}+(d^{3}-2)g^{2}-(d^{3}-2d^{2}-1)g-(2d-2)
\end{align*}
after applying the class computations in Lemma \ref{subordinate} and Propositions \ref{small_diag} and \ref{other_diag}, and we have the solution 
\begin{align*}
a &= g^{3}-(2d-1)g^{2}+(d^{2}-2)g-(d-1)(d-2)=(g-d+1)(g^{2}-dg+(d-2))\\
b &= g^{3}+(2-2d)g^{2}+(d^{2}-2d-1)g+(2d-2)=(g-d+1)(g^{2}-(d-1)g-2).
\end{align*}
\end{proof}

We now turn to the proof of Theorem F.  Before introducing the relevant divisor, we prove two preliminary lemmas.
\begin{lem}
\label{combsum}
For all $m \geq 1,$
$$(2m+3)\cdot\displaystyle\sum_{l=0}^{m}(-1)^{l}(l+1)\displaystyle\binom{2m-l}{m}\displaystyle\binom{2m+2}{l+3}=-(m+2)\cdot\displaystyle\sum_{l=0}^{m}(-1)^{l}l(l+1)\displaystyle\binom{2m-l}{m}\displaystyle\binom{2m+3}{l+3}$$
\end{lem}
\begin{proof}
It suffices to show that for all $m \geq 1,$ $$\displaystyle\sum_{l=0}^{m}(-1)^{l}(l+1)\displaystyle\binom{2m-l}{m}\Bigl(l(m+2)\displaystyle\binom{2m+3}{l+3}+(2m+3)\binom{2m+2}{l+3}\Bigr)=0.$$
First note that $$l(m+2)\displaystyle\binom{2m+3}{l+3}+(2m+3)\displaystyle\binom{2m+2}{l+3}=((m+1)l+2m)\displaystyle\binom{2m+3}{2m-l}$$ and  $$\displaystyle\binom{2m-l}{m}\displaystyle\binom{2m+3}{2m-l}=\displaystyle\binom{2m+3}{m}\displaystyle\binom{m+3}{m-l}$$ for all $l,$ so that the left-hand side is proportional to the sum $$\displaystyle\sum_{l=0}^{m}(-1)^{l}\bigl((m+1)l^{2}+(3m+1)l+2m\bigr)\displaystyle\binom{m+3}{m-l}.$$  This is the $m$-th convolution of the sequences $\bigl\{(-1)^{l}((m+1)l^{2}+(3m+1)l+2m)\bigr\}_{l}$ and $\bigl\{\binom{m+3}{l}\bigr\}_{l}$ whose respective generating functions are $\frac{2m-2t}{(1+t)^{3}}$ and $(1+t)^{m+3}.$  Therefore it is the coefficient of $t^{m}$ in $(2m-2t)(1+t)^{m},$ which is zero.  
\end{proof}

\begin{lem}
\label{orth}
If $|\mathcal{L}|$ is any pencil of degree $k+1$ on $C,$ then $\Gamma_{k}(\mathcal{L}) \cdot (\theta-(2-\frac{1}{k})x)=0.$
\end{lem}
\begin{proof}
Fix a pencil $|\mathcal{L}|$ of degree $k+1$ on $C.$  It suffices to show that $\Gamma_{k}(\mathcal{L}) \cdot \theta=2k-1$ and $\Gamma_{k}(\mathcal{L}) \cdot x = k.$  By Lemmas \ref{intersect} and \ref{subordinate}, we have that 
\begin{align*}
\Gamma_{k}(\mathcal{L}) \cdot \theta &= (2k-1) \cdot \displaystyle\sum_{j=0}^{k-1}(-1)^{j}\displaystyle\binom{k-2+j}{j}\displaystyle\binom{2k-2}{k-1-j}\\ 
\Gamma_{k}(\mathcal{L}) \cdot x &= \displaystyle\sum_{j=0}^{k-1}(-1)^{j}\displaystyle\binom{k-2+j}{j}\displaystyle\binom{2k-1}{k-1-j}
\end{align*}  
Since the sequences  $\bigl\{(-1)^{j}\binom{k-2+j}{j}\bigr\}_{j},$ $\bigl\{\binom{2k-2}{j}\bigr\}_{j},$ and $\bigl\{\binom{2k-1}{j}\bigr\}_{j}$ have respective generating functions $(1+t)^{-(k-1)}$, $(1+t)^{2k-2}$, and $(1+t)^{2k-1}$, arguing as in the proof of Lemma \ref{combsum} establishes the desired equalities. 
\end{proof}  

Let $C$ be a general curve of genus $2k-1 \geq 5.$  Consider the diagram
$$\xymatrix{
C_{k} \times C_{2k-5} \ar[d]^{\pi} \ar[r]^{\sigma} &C_{3k-5}\\
C_{k}}$$ where $\sigma$ is defined by $\sigma(D,E)=D+E$ and $\pi$ is projection.  Since $C$ is general, we may deduce from Theorem \ref{gieseker} that $C^{k-2}_{3k-5}$ is $(k-1)-$ dimensional.  Furthermore, for each $D \in C_{k}$ we have that $\pi^{-1}(D) \cap \sigma^{-1}(C^{k-2}_{3k-5})$ is at most finite.  It follows that the cycle $E_{(k)}:=\pi_{\ast}\sigma^{\ast}(C^{k-2}_{3k-5})$ is an effective divisor on $C_{k}.$ 

\medskip
  
\textit{Proof of Theorem F:}  (i) By the push-pull formulas (e.g. Exercise D-8 in \cite{ACGH}) the class of $E_{(k)}$ is $$\displaystyle\frac{1}{k-1} \cdot \Bigl(\Bigl(\displaystyle\sum_{l=0}^{k-2}(-1)^{l}(l+1)\binom{2k-4-l}{k-2}\binom{2k-2}{l+3}\Bigr)\theta$$ 
$$+\Bigl(\displaystyle\sum_{l=0}^{k-2}(-1)^{l}l(l+1)\binom{2k-4-l}{k-2}\binom{2k-1}{l+3}\Bigr)x\Bigr)$$  Setting $m=k-2$ in Lemma \ref{combsum}, we see that this class is proportional to $\theta-(2-\frac{1}{k})x.$

(ii) Let $D$ be an effective divisor on $C_{k}$ whose class is proportional to $\theta-tx$ for some $t > 2-\frac{1}{k}.$  By Lemma \ref{orth}, $D \cdot \Gamma_{k}(\mathcal{L})<0$ for all $\mathcal{L} \in W^{1}_{k+1}(C),$ so that $Z_{k} \subseteq \textbf{B}(D).$ 

(iii) We employ a slight variation on an argument used in the proof of Theorem 5 of \cite{Kou}.  First we show that $E_{(3)}$ is irreducible.  We have a morphism $\mu: C \times W^{1}_{4}(C) \rightarrow C_{3}$ which takes each pair $(p,\mathcal{L})$ to the unique element of the linear system $|\mathcal{L}(-p)|.$  Since $C$ is general, the Brill-Noether locus $W^{1}_{4}(C)$ is a smooth irreducible curve, which implies the irreducibility of $C \times W^{1}_{4}(C).$  The irreducibility of $E_{(3)}$ then follows at once from the observation that $E_{(3)}$ is the image of $\mu.$

It follows from (i) that the class of $E_{(3)}$ is proportional to $\theta-\frac{5}{3}x.$  Suppose that there is an irreducible effective divisor $D'$ whose class is proportional to $\theta-tx$ for $t>\frac{5}{3}.$  Then for any $\mathcal{L} \in W^{1}_{4}(C)$ we have $\Gamma_{3}(\mathcal{L}) \cdot D' < 0,$ so that all such curves are contained in $D'.$  But this means $D'=E_{(3)},$ which is absurd. \hfill \qedsymbol

\subsubsection{Results on Base Loci and Volume} 
\label{secbas}
The next theorem implies (ii) of Theorem A when combined with Lemma \ref{char_cones} and Theorem 3 in \cite{Kou}, and it implies Theorem C when combined with Proposition \ref{plucker_ample}.

\begin{thm}
\label{stab_classes}
Let $\mathcal{L}$ be a stable line bundle on $G^{g-d-1}_{2g-2-d}(C)$ with stable base locus $Z$ and numerical class $a\widehat{X}+b\widehat{\theta}$ satisfying $a>0$ and $a+b>0.$ Then the line bundle $({\tau}^{-1})^{\ast}\mathcal{L}$ on $C_{d}$ is stable with stable base locus $C^{1}_{d} \cup \tau^{-1}(Z).$  
\end{thm}
\hspace{0.6cm}In particular, a line bundle $\mathcal{L}$ on $G^{g-d-1}_{2g-2-d}(C)$ with numerical class in the aforementioned range is stable precisely when the pullback bundle $({\tau}^{-1})^{\ast}\mathcal{L}$ on $C_{d}$ is stable.

\begin{proof}
Let $\mathcal{L}$ be a stable line bundle on $G^{g-d-1}_{2g-2-d}(C)$ satisfying the hypotheses, and let $\mathcal{M}:=(\tau^{-1})^{\ast}\mathcal{L}.$  By Proposition \ref{multiple}, we may assume without loss of generality that $\textrm{Bs}(|\mathcal{L}|)=\textbf{B}(\mathcal{L})=Z$ and $\textrm{Bs}(|\mathcal{M}|)=\textbf{B}(\mathcal{M}).$

\hspace{0.6cm}The hypothesis on the coefficients $a$ and $b$ guarantees that the numerical class of $\mathcal{M},$ which is $(a+b)\theta-ax,$ lies in the fourth quarter of the $\theta,x-$plane, so that the stable base locus of $\mathcal{M}$ must contain $C^{1}_{d}.$  By Proposition \ref{picard_isom}, pullback via ${\tau}^{-1}$ gives a natural isomorphism between $H^{0}(G^{g-d-1}_{2g-2-d}(C),\mathcal{L})$ and $H^{0}(C_{d},\mathcal{M}),$ so $\textbf{B}(\mathcal{M})=\textrm{Bs}(|\mathcal{M}|)=C^{1}_{d} \cup \tau^{-1}(Z).$  Indeed, if $x \in C_{d}-C^{1}_{d}$ is a basepoint of $|\mathcal{M}|,$ then $\tau(x)$ is a basepoint of $|\mathcal{L}|.$ 

\hspace{0.6cm}The set of stable classes in in $N^{1}_{\mathbb{R}}(G^{g-d-1}_{2g-2-d}(C))$ having $Z$ as its stable base locus is open, so its image under $(\tau^{-1})^{\ast}$ is open as well.  If $t_{0}:=\frac{a}{a+b},$ then by our previous calculation and Propositions \ref{stab_dense} and \ref{stab_drop}, we have that for some $\epsilon > 0,$ $\theta-tx$ is stable with stable base locus $C^{1}_{d} \cup \tau^{-1}(Z)$ for all $t$ satisfying $0<|t-t_{0}|<\epsilon.$  We then have by the definitions of the augmented and restricted base loci that $\textbf{B}_{-}( \mathcal{M})=C^{1}_{d} \cup \tau^{-1}(Z)=\textbf{B}_{+}(\mathcal{M}).$
\end{proof}    

We now prove Theorems D and E.  The following result is a rephrasing of Lemma 2.2 in \cite{Pac}; we refer to \textit{loc. cit.} for its proof.

\begin{lem}
\label{nef_small_diag}
If $C$ is any curve of genus $g$ and $d \geq 3,$ the numerical class $-\theta+dgx$ in $N^{1}_{\mathbb{R}}(C_{d})$ is nef and big, and its augmented base locus is $\Delta_{d}.$  In particular, $-\theta+dgx$ spans a boundary ray of the nef cone of $C_{d}.$ \hfill \qedsymbol
\end{lem}

\textit{Proof of Theorem D:}  Let $C$ be a general curve of genus $g \geq 5.$  By Propositions \ref{stab_dense} and \ref{stab_drop} and Lemma \ref{nef_small_diag}, there exists a positive $t_{0}<g^{2}-g$ such that the class $-\theta+tx$ on $C_{g-1}$ is stable with stable base locus $\Delta_{g-1}$ whenever $t_{0}<t<g^{2}-g.$  It then follows from Theorem \ref{stab_classes} and Proposition \ref{explicit_isom} that $\widetilde{\tau}^{\ast}(-\theta+tx)=(t-1)\theta-tx$ is stable with stable base locus $C^{1}_{g-1} \cup \widetilde{\tau}^{-1}(\Delta_{g-1}).$  Since a given curve has only finitely many Weierstrass points, $\widetilde{\tau}^{-1}(\Delta_{g-1}) \nsubseteq C^{1}_{g-1}.$  Finally, since $g \geq 5,$ the dimension of $C^{1}_{g-1}$ is at least 2, so that $C^{1}_{g-1} \cup \widetilde{\tau}^{-1}(\Delta_{g-1})$ is non-equidimensional. \hfill \qedsymbol  

\medskip

When $C$ is general, all its Weierstrass points are ordinary (e.g. \cite{EisHar}), so that $\widetilde{\tau}^{-1}(\Delta_{g-1})$ and $C^{1}_{g-1}$ are disjoint.  However, it is possible for the stable base locus in the previous proof to be connected.  The locus $\mathfrak{W}$ in $\mathcal{M}_{g}$ parametrizing curves with a Weierstrass point having gap sequence $\{1, \cdots ,g-2,g,g+1\}$ is of codimension 1, and since we are assuming $g \geq 5$ the hyperelliptic locus in $\mathcal{M}_{g}$ has codimension at least 3.  Thus if the isomorphism class of $C$ is a general point in $\mathfrak{W},$ the above proof is valid for $C,$ and $\widetilde{\tau}^{-1}(\Delta_{g-1}) \cap C^{1}_{g-1} \neq \emptyset.$        

\medskip

\textit{Proof of Theorem E:} It is an immediate consequence of (ii) in Proposition \ref{picard_isom} that for any line bundle $\mathcal{L}$ on $C_{g-1}$ we have $\textnormal{vol}_{C_{g-1}}(\mathcal{L})=\textnormal{vol}_{C_{g-1}}(\widetilde{\tau}^{\ast}\mathcal{L}).$  The result then follows from applying Lemmas \ref{intersect}, \ref{vol_nef}, and \ref{nef_small_diag}.  \hfill \qedsymbol

\subsection{The Hyperelliptic Case}
\label{hyp}
Let $C$ be a general hyperelliptic curve of genus $g \geq 2.$  Note that since $C^{1}_{d}$ is neither empty nor of the expected dimension $$(g-2(g-d+1))+1=2d-(g+1)$$ the computation of the class of $C^{1}_{d}$ given on p.326 of \cite{ACGH} does not apply.

\medskip

\textit{Proof of Theorem G:} \textit{(i)} We denote the hyperelliptic pencil on $C$ by $|\mathcal{L}|.$  By Clifford's Theorem, the dimension of the linear series $|{\mathcal{L}}^{\otimes (d-1)}|$ is $d-1.$  We have from Lemma \ref{subordinate} that $\Gamma_{d}(\mathcal{L}^{\otimes (d-1)})$ is a divisor on $C_{d}$ and that its class is $\theta-(g-d+1)x.$  For any $D \in C_{d},$ it follows from Riemann-Roch that $\dim{|D|} \geq 1$ if any only if $$\dim|\mathcal{L}^{\otimes (d-1)}(-D)| \geq 0.$$  This equivalence of algebraic conditions gives the equality of cycles $$C^{1}_{d}=\Gamma_{d}(\mathcal{L}^{\otimes (d-1)}).$$  Since $C^{1}_{d}$ is the exceptional locus of the divisorial contraction $a_{0,d},$ it is not big; therefore its class spans a boundary of the effective cone of $C_{d}.$ 

\smallskip

\textit{(ii):}  Since $\theta$ is nef and big for $2 \leq d \leq g-1,$ this is a consequence of Theorem \ref{nak_base_loc}.

\smallskip

\textit{(iii):}  For $t \in (0,g-d+1],$ $$\textrm{vol}_{C_{d}}(\theta-tx)=\textrm{vol}_{C_{d}}{\Bigl(}(1-\frac{t}{g-d+1})\theta+(\frac{t}{g-d+1})(\theta-(g-d+1)x){\Bigr)}$$ $$=\Bigl(\frac{t}{g-d+1}\Bigr)^{d} \cdot \textrm{vol}_{C_{d}}\Bigl((\frac{g-d+1}{t}-1)\theta+(\theta-(g-d+1)x)\Bigr).$$  

By (ii), $\theta$ spans a boundary of the movable cone, so that $(\frac{g-d+1}{t}-1)\theta$ is the positive part of a Zariski decomposition $(\frac{g-d+1}{t}-1)\theta+(\theta-(g-d+1)x).$  It then follows from Proposition 3.20 in \cite{Bou} and Lemma \ref{vol_nef} that $$\Bigl(\frac{t}{g-d+1}\Bigr)^{d} \cdot \textrm{vol}_{C_{d}}\Bigl((\frac{g-d+1}{t}-1){\theta}+(\theta-(g-d+1)x)\Bigr)=$$ $$\Bigl(\frac{t}{g-d+1}\Bigr)^{d} \cdot \Bigl(\frac{g-d+1}{t}-1\Bigr)^{d} \cdot \frac{g!}{(g-d)!}=\frac{g!}{(g-d)!}\Bigl(1-\frac{t}{g-d+1}\Bigr)^{d}.$$ \hfill \qedsymbol

\medskip

\vspace{0.1in}
\begin{flushleft}
Department of Mathematics, Stony Brook University\\
Stony Brook, NY,11794-3651\\
\end{flushleft}

\medskip

\vspace{0.1in}
\begin{flushleft}
current address: Department of Mathematics, University of Michigan\\
2074 East Hall, 530 Church Street\\
Ann Arbor, MI 48109-1043\\
\end{flushleft}

\end{document}